\documentclass[11pt]{article}
\usepackage{amsmath, amssymb, amsthm}
\usepackage{enumitem}                                 
\usepackage[margin=1.0in]{geometry}
\usepackage{color}

\AtEndDocument{
\bigskip{\footnotesize%
\textsc{Department of Mathematics, The University of Haifa at Oranim, Tivon 36006, Israel} \par 
  \textit{E-mail address} \texttt{ amir.algom@math.haifa.ac.il}
} 

\bigskip{\footnotesize%
\textsc{Department of Mathematical Sciences, P.O. Box 3000, 90014 University of Oulu, Finland} \par 
  \textit{E-mail address} \texttt{  meng.wu@oulu.fi
} 
}}

\newtheorem{theorem}{Theorem}[section]

\newtheorem{Counter-example}[theorem]{Counter example}
\newtheorem{Claim}[theorem]{Claim}

\newtheorem{Lemma}[theorem]{Lemma}
\newtheorem{Proposition}[theorem]{Proposition}

\newtheorem{Definition}[theorem]{Definition}

\newtheorem*{theorem*}{Theorem}

\DeclareMathOperator*{\cmpct}{cpct}

\DeclareMathOperator*{\diag}{diag}

\newcommand{\diam}{\text{diam}}

\usepackage{amsthm}

\newtheorem{innercustomgeneric}{\customgenericname}
\providecommand{\customgenericname}{}
\newcommand{\newcustomtheorem}[2]{%
  \newenvironment{#1}[1]
  {%
   \renewcommand\customgenericname{#2}%
   \renewcommand\theinnercustomgeneric{##1}%
   \innercustomgeneric
  }
  {\endinnercustomgeneric}
}

\newcustomtheorem{customthm}{Theorem}
\newcustomtheorem{customlemma}{Lemma}

\makeatother

\title{Large slices through self affine carpets}

\author{Amir Algom and Meng Wu}
\date{}

\begin{document}
\maketitle
\abstract{Let $F\subseteq [0,1]^2$ be a Bedford-McMullen carpet defined by exponents $m>n$, that projects to $[0,1]$ on the $y$-axis. We show that under mild conditions on $F$,  there are many non principle lines $\ell$ such that $\dim^* F\cap \ell = \dim^* F -1$, where $\dim^*$ is Furstenberg's star dimension (maximal dimension of a microset). This exhibits the sharpness of recent Furstenberg-type slicing theorems obtained by Algom (2020) about  upper bounds on the dimension of every such slice. 
 }

\section{Introduction} \label{introduction}
Let $F\subset \mathbb{R}^2$ be a set and let $\ell \subset \mathbb{R}^2$ be an affine line. In this paper we consider the classical question of estimating the dimension of $F\cap \ell$ in terms of the dimension of $F$. For $(u,t)\in  \mathbb{R} \times \mathbb{R}$ let $\ell_{u,t}$ denote the planar line with slope $u$ that intersects the $y$-axis at $t$ (notice that we exclude from notation lines that are parallel to the $y$-axis). By Marstrand's slicing Theorem, for any fixed slope $u$,
\begin{equation} \label{Eq. Marstrand}
\dim_H F\cap \ell_{u,t} \leq \max \lbrace \dim_H F -1, 0 \rbrace \text{ for Lebesgue almost every } t, 
\end{equation}
where $\dim_H F$ denotes the Hausdorff dimension of the set $F$. This is known to fail for any smaller value on the right hand side of  \eqref{Eq. Marstrand}.

While \eqref{Eq. Marstrand} predicts the dimension of the intersection of $F$ with a typical line $\ell$, it is  a  challenging problem to understand the intersection of $F$ with a fixed line $\ell$. Nonetheless,  in recent years there has been significant progress towards finding sharper versions of \eqref{Eq. Marstrand} when the underlying set $F$ has some arithmetic or dynamical origin. 

We will focus our attention on one such class of sets,  Bedford-McMullen carpets. These carpets are defined as follows: let $m> n>1$, and let 
\begin{equation*}
D \subseteq \lbrace 0,...,m-1 \rbrace \times \lbrace 0,...,n-1 \rbrace.
\end{equation*}
We then define
\begin{equation*}
F = \left\{ \left(\sum_{k=1} ^\infty \frac{x_k}{m^k}, \sum_{k=1} ^\infty \frac{y_k}{n^k}\right) :\quad  (x_k,y_k) \in D \right\}.
\end{equation*}
The set $F$ is  called a Bedford-McMullen carpet with defining exponents $m,n$, and allowed digit set $D$.

Recently it  has been shown that when $\frac{\log m}{\log n}\notin \mathbb{Q}$  these carpets  satisfy strong versions of Marstrand's  slicing Theorem, that hold for \textit{all} lines not parallel to the major axes. To state these results, let us recall some notions.  For a set $X \subseteq [0,1]^d$ we denote by  $\dim^* X$ its star dimension, 
\begin{equation} \label{Eq star dim}
\dim^* X:= \sup \lbrace \dim_H M: \quad M  \text{ is a microset of } X \rbrace.
\end{equation}
Recall that microsets of $X$ are limits in the Hausdorff metric on subsets of $[-1,1]^2$ of "blow-up" of increasingly small balls about points in $X$ (for a formal definition of a microset see Section \ref{Section microsets}). An explicit formula  for the star-dimension of $F$ in terms of $D,m,n$ was given by Mackay  \cite{mackay2011assouad}. This notion was originally introduced and stuided by Furstenberg in \cite{furstenberg2008ergodic}; In our setting it coincides with the notion of Assouad dimension \cite{fraser2020book}. However,  for consistency with the recent literature on the subject of slicing theorems we work here with star-dimension. 
\begin{customthm}{$A$}  \cite{algom2018slicing} \label{Theorem known}
Let $F$ be a Bedford-McMullen carpet with exponents $(m,n)$. If $\frac{\log m}{\log n} \notin \mathbb{Q}$ then for every $u\neq 0$ and $t\in \mathbb{R}$
\begin{equation} \label{Eq LHS}
\dim^* (\ell_{u,t} \cap F) \leq   \max \lbrace \dim^* F -1,0 \rbrace.
\end{equation}
\end{customthm}

The first version of Theorem \ref{Theorem known} was  proved simultaneously and independently by Shmerkin \cite{shmerkin2016furstenberg} and Wu \cite{wu2016proof} when $F$ is a product set. This  result led directly to the resolution of Furstenberg's slicing Conjecture \cite{furstenberg1970intersections}.  It extended previous work due to Furstenberg himself  \cite{furstenberg1970intersections}, Wolff \cite{Wolff1999Kakeya}, and Feng, Huang, and Rao \cite{feng2014affine}. A simple proof of the Conjecture was later found by Austin \cite{austin2020new}, and some improvments were given by Yu \cite{Yu2021Wu}.   Theorem \ref{Theorem known} in full generality was obtained later by Algom  \cite{algom2018slicing}, by extending Wu's method \cite{wu2016proof}.  We also note that B\'{a}r\'{a}ny,   K\"{a}enm\"{a}ki, and Yu \cite{Barany2021finer}, recently obtained similar results about slices through some non-carpet planar self-affine sets.

The main goal of this paper is to study the sharpness of Theorem \ref{Theorem known}. It is well known that Theorem \ref{Theorem known} is sharp when $F$ is Ahlfors regular, regardless of any arithmetic assumptions on $m,n$; Indeed this follows by combining the standard facts that here  $\dim^* F = \dim_H F$, that for every compact set $\dim_H X \leq \dim^* X$, and that if $\dim_H F>1$ there are slices through $F$ whose Hausdorff dimension approaches $\dim_H F-1$ \cite{Mattila2015Fourier}. In this work we will exhibit a large class of carpets that are not Ahlfors regular where Theorem \ref{Theorem known} remains {\em sharp}.

We emphasize that our focus here is on the sharpness of \eqref{Eq LHS} \textit{as stated}: When considering in \eqref{Eq LHS} other notions of dimension for the both the slice and the carpet then \eqref{Eq LHS} is no longer sharp. Indeed, in \cite{Algom2021Wu}, we showed that  if $\frac{\log m}{\log n} \notin \mathbb{Q}$ then  for all $u\neq 0$ and $t\in \mathbb{R}$,
\begin{enumerate}
\item $ \dim_H (\ell_{u,t} \cap F) \leq \max \left\lbrace 0, \, \frac{\dim_H F}{\dim^* F} \cdot (\dim^* F-1) \right\rbrace$.

\item $\dim_P (\ell_{u,t} \cap F) \leq \max \left\lbrace 0, \, \frac{\dim_P F}{\dim^* F} \cdot (\dim^* F-1) \right\rbrace$, where $\dim_P X$ denotes the packing dimension of a set $X$.
\end{enumerate}
Since  for non Ahlfors regular carpets we have  $\dim_H F < \dim_P F < \dim^* F$ (see e.g. \cite[Chapter 4]{bishop2013fractal}),  these results strictly  improve \eqref{Eq LHS} for both the packing and the Hausdorff dimension. However, the best possible upper bounds  remain unknown  (see e.g. Fraser's question   \cite[Question 8.3]{fraser2020fractal}).

Let us now state our main result: Given $r>0$ and $x\in \mathbb{R}$ we denote by $B(x,r)$ the open ball about $x$ with radius $r$. Also, denote the coordinate projections $P_1(x,y)=x$ and $P_2(x,y)=y$.

\begin{theorem} \label{Main Theorem}
Let $F$ be a Bedford-McMullen carpet with exponents $m>n$ such that there exists $i_0\in  \lbrace 0,...,m-1 \rbrace $ satisfying 
$$\lbrace j\in \lbrace 0,...,n-1 \rbrace:\quad  (i_0,j)\in D\rbrace=  \lbrace 0,...,n-1 \rbrace.$$
Then there  exists  $r=r(F)>0$ such that for every $u\in B(0,r)$ there is some $t\in \mathbb{R}$ with
$$\dim^* \ell_{u,t} \cap F =  \dim^* F-1.$$
\end{theorem}

Notice that the assumption made on $F$ implies that there exists $x_0\in[0,1]$ such that 
$$ \lbrace x_0 \rbrace \times [0,1] \subseteq F.$$
In particular, $P_2(F)=[0,1]$ and so $\dim^* F \geq 1$ by \cite{mackay2011assouad}. It is an interesting question whether it suffices to assume that  $P_2(F) = [0,1]$, or even just that $\dim^* F \geq 1$, in order for Theorem \ref{Main Theorem} to hold true. We leave this to future research.

Theorem \ref{Main Theorem} is also related to a question raised by Shmerkin \cite[Section 8  Remark (b)]{shmerkin2016furstenberg}: Given a set $E\subseteq \mathbb{R}^2$ with $\dim_H E>1$, show that for many (in some sense) pairs $(u,t)$ one has $\dim_H E\cap \ell_{u,t} \approx \dim_H E-1$. Theorem \ref{Main Theorem} offers some progress towards this goal, though for star dimension rather than Hausdorff dimension.

We end this introduction with a sketch of our method. Let $F$ be a Bedford-McMullen carpet as in Theorem \ref{Main Theorem}. We will prove that for all $ |u|$ small enough  there exists some   $t\in \mathbb{R}$ and $A,B>0$ such that the following holds: Letting $\mathcal{D}_{m^p}$ denote the $m^p$-adic partition of $\mathbb{R}^2$ (see Section  \ref{Section microsets}),  for every $i\in \mathbb{N}$  there are  integers $k_i\geq 0$ and cells $D_i \in \mathcal{D}_{m^{k_i}}$ such that
\begin{equation} \label{Eq claimed line}
N(\ell_{u,t}\cap F \cap D_i,\, \mathcal{D}_{m^{k_i+i}})\geq A\cdot \left(  m^{(\dim^* F -1)\cdot i}\right)-B,
\end{equation}
where $N(X,\, \mathcal{D}_{m^p})$ is the corresponding covering number of a set $X$. Via Furstenberg's formula for star dimension  (see Theorem \ref{Furstneberg Formula} below) this implies that $\dim^* \ell_{u,t}\cap F \geq \dim^* F -1$. Since $P_2(F)=[0,1]$ the upper bound $\dim^* \ell_{u,t}\cap F \leq \dim^* F -1$ always holds true (see Proposition \ref{Prop upper bound} below), so all in all we obtain $\dim^* \ell_{u,t}\cap F = \dim^* F -1$.

Our basic observation is that our assumption on $F$ implies that for some small $c_0=c_0(F)>0$, if $|u|\leq c_0$ then the projection of $F$ in the direction transverse to the line $\ell_{u,0}$ is an interval of size $1-o(1)$. This is proved in Lemma \ref{lemma:projection-condition-1}. When $|u|\leq c_0$ we derive from this fact  two important consequences: First, that if $t$ and $Q\in \mathcal{D}_{m^k}$ satisfy  $\ell_{u,t} \cap F\cap D \neq \emptyset$,  then, given $b\in \mathbb{N}$,  if $t'$ and $t$ are very close then
\begin{equation} \label{Eq 1}
N\left(\ell_{u,t} \cap F\cap Q, \, \mathcal{D}_{m^{k+b}} \right) \approx  N\left(\ell_{u,t'} \cap F\cap Q, \, \mathcal{D}_{m^{k+b}} \right).
\end{equation}
This is proved in Claim \ref{Lemma 2}.  Secondly, let $t\in [0,1]$ be such that $\dim_H  \ell_{0,t}\cap F = \dim^* F-1$ (the existence of such a $t$ follows from \cite{mackay2011assouad}). Then for every $p\in \mathbb{N}$, for all $|u|\ll 1$ in a manner dependent on $p$,  there is a small open neighbourhood (that depends on all previous parameters) such that every $t'$ there satisfies
\begin{equation} \label{Eq 2}
N(\ell_{u,t'}\cap F , \, \mathcal{D}_{m^p} ) \approx   N(\ell_{0,t}\cap F , \, \mathcal{D}_{m^p}) = m^{(\dim^* F-1)\cdot p}.
\end{equation}
Note that it is the left "$\approx$" that is of interest here; the other equality is well known. This is proved in Claim \ref{Claim hori line}.

Using these two estimates on the effects of small perturbations of  slope and intercept on the covering numbers,  the  line as in \eqref{Eq claimed line} is obtained as a (Hausdorff metric) limit of a sequence of lines. These lines are constructed inductively: The first line is constructed so that \eqref{Eq 2} holds for $p=1$, and its slope $|u|<c_0$  so that \eqref{Eq 1} holds. For the inductive step, we assume all previous lines  have the same slope $u$ as in the first step. The next line is constructed so that it satisfies \eqref{Eq 2} for $p$, and so that its intercept lies in small enough neighbourhoods of \textit{all} the intercepts of the previous lines, so that \eqref{Eq 1} may be applied. Here (and in many other proofs in this paper) the self-affine structure of $F$ is used in a crucial way. Note that in this informal discussion we did not disclose how the integers $k_i\geq 0$ and cells $D_i \in \mathcal{D}_{m^{k_i}}$ from \eqref{Eq claimed line} are obtained; We refer the reader to Claim \ref{Key claim} for the full details of this construction.

\medskip

\noindent{\textbf{Acknowledgement}} The authors  thank Mike Hochman  for his remarks on previous versions of this manuscript.

\section{On the proof of Theorem \ref{Main Theorem}} \label{Section opt}
\subsection{Bedford-McMullen carpets} \label{Section approximate squares}

Let $\Phi = \lbrace \phi_i \rbrace_{i\in\Lambda} $ be a family of contractions $\phi_i : \mathbb{R}^d \rightarrow \mathbb{R}^d, d\geq 1$, where $\Lambda$ is a finite alphabet set. The family $\Phi$ is called an iterated function system (IFS). It is well known that there exists a unique compact $\emptyset \neq F \subset \mathbb{R}^d$ such that $F = \bigcup_{i\in\Lambda} \phi_i (F)$. $F$ is called the attractor of $\Phi$, and $\Phi$ is called a generating IFS for $F$.  The set of finite words over $\Lambda$ is denoted $\Lambda^*$,  i.e.,  $\Lambda^*=\cup_{n\ge 1}\Lambda^n$.  For a multi-index $I=(i_1,...,i_k) \in \Lambda^*  $, we define its length  $|I|\in \mathbb{N}$ by $k$, and write 
\begin{equation*}
\phi_I := \phi_{i_1} \circ ... \circ \phi_{i_k}.
\end{equation*}
The map $\phi_I$ is called a cylinder map of the IFS, whereas the set $\phi_I (F)$ is called a cylinder set of $F$.  If $I= (i_1,i_2, ... )\in\Lambda^\mathbb{N}$ is infinite, we define $\phi_I \in F$ by  
$$\phi_I := \lim_{k\to\infty}\phi_{i_1} \circ ... \circ \phi_{i_k}(0).$$
Finally, a set $F \subset \mathbb{R}^d$ is  called self affine if there exists a generating IFS $\Phi$ for $F$ such that $\Phi$ consists only of affine mappings.

Next, recall the definition of a Bedford-McMullen carpet $F$ with defining exponents $m,n$ and allowed digit set $D$ from Section \ref{introduction}. Notice  that $F$ is a self affine set generated by an IFS consisting of maps whose linear parts are diagonal matrices. Specifically, $F$ is the attractor of $\Phi = \lbrace \phi_{(i,j)} \rbrace_{(i,j) \in D}$ where
\begin{equation} \label{Genrating IFS for F}
\phi_{(i,j)} (x,y) = \left(\frac{x+i}{m}, \frac{y+j}{n}\right) = \begin{pmatrix}
\frac{1}{m} & 0 \\
0 & \frac{1}{n}
\end{pmatrix}
\cdot (x,y) + \left(\frac{i}{m},\frac{j}{n}\right).
\end{equation}

 \subsection{Star dimension  and covering numbers}  \label{Section microsets}  
Let $X$ be a compact metric space. Let $\cmpct(X)$ denote the set of non-empty closed subsets of $X$. For $A,B\in\cmpct(X)$ and $\epsilon >0$ define
\begin{equation*}
A_\epsilon = \lbrace x \in X : \quad \exists a\in A, d(x,a) < \epsilon \rbrace.
\end{equation*}
The Hausdorff distance between $A$ and $B$ is defined by
\begin{equation*}
d_H (A,B) = \inf \lbrace \epsilon >0 : \, A \subseteq B_\epsilon, \quad B \subseteq A_\epsilon \rbrace. 
\end{equation*}
Endowed with this metric, $\cmpct(X)$ becomes a compact metric space (see e.g. the appendix in \cite{bishop2013fractal}).  

Now,  let us restrict to $X=[-1,1]^2$. Let $F\subseteq [0,1]^2$  be a compact set. A set $A \subseteq [-1,1]^2$ is called a  miniset of $F$ if $A \subseteq (a \cdot F + t)\cap [-1,1]^2$ for some $a \geq 1, t\in \mathbb{R}$. A set $M$ is called a  microset of $F$ if $M$ is a limit (in the  Hausdorff metric) of minisets of $F$. Let $\mathcal{G}_F$ denote the family of all microsets of $F$. Recall, from \eqref{Eq star dim}, that the star dimension of $F$ is the defined as
\begin{equation*}
\dim^* F = \sup \lbrace \dim_H A : \, A\in \mathcal{G}_F \rbrace.
\end{equation*}

Alternatively, one may compute $\dim^* F$ using a formula due to Furstenberg \cite{furstenberg2008ergodic}. To this end, recall that if $X\subset \mathbb{R}^i$, $i=1,2$ is a bounded set and $\mathcal{G}$ is some  partition of $\mathbb{R}^i$, then the covering number of $X$ with respect to $\mathcal{G}$ is defined as
\begin{equation*}
N(X,\, \mathcal{G}) := | \lbrace D\in \mathcal{G}:\, D\cap X\neq \emptyset \rbrace|.
\end{equation*}
In all our applications this will be a finite number. 

Consider the partition $\mathcal{D}_{m^p} $ of $\mathbb{R}$, where $m>1,p \geq 1$ are integers, which is defined as
$$\left\lbrace \left[\frac{k}{m^p}, \frac{k+1}{m^p}\right):\, k\in \mathbb{Z}\right\rbrace.$$
Then $\mathcal{D}_{m^p} \times \mathcal{D}_{m^p}$ forms a partition of $\mathbb{R}^2$. We will usually abuse notation and denote this partition by $\mathcal{D}_{m^p}$ as well (which partition is meant will be clear from context). Here is Furstenberg's formula for star dimension:
\begin{theorem} \label{Furstneberg Formula} (Furstenberg, \cite{furstenberg2008ergodic})
Let $X\subseteq [-1,1]^2$ be a compact non-empty set, and let $m\geq 2$ be an integer. Then
\begin{equation*}
\dim^* X = \lim_{i\rightarrow \infty} \max_{k\in \mathbb{N}} \left \lbrace \frac{\log N(X\cap D, \,  \mathcal{D}_{m^{k+i}})}{i\log m} : D\in \mathcal{D}_{m^k} \right \rbrace.
\end{equation*}
\end{theorem}

We will require the following definition:
\begin{Definition} \label{Def adj}
We say that two cells $D,D' \in \mathcal{D}_{m^p}$ are adjacent if their closures intersect non-trivially. 
\end{Definition}
In particular, this means that every $D$ is adjacent to itself.

The following elementary Lemma is about a certain type of continuity covering numbers posses with respect to the Hausdorff metric:

\begin{Lemma} \label{Lemma 0}
Let $X_k \rightarrow X$ in the Hausdorff metric on compact sets of $[0,1]^2$, and let $m\geq 2$ be an integer. Then for every $p\geq 3$ there is some $N=N(m,p)\in \mathbb{N}$ such that for every $k>N$
\begin{equation*}
N(X, \,\mathcal{D}_{m^p}) \geq \frac{1}{9}\cdot N(X_k,\, \mathcal{D}_{m^p}).
\end{equation*}
\end{Lemma}
\begin{proof}
Let $N$ be large enough so that $d_H (X_k,\,X)\leq \frac{1}{m^{2p}}$ for every $k>N$. For every $k>N$ and   $D\in \mathcal{D}_{m^p}$  such that $D\cap X_k\neq \emptyset$, let $x_D \in D\cap X_k$ be some point. Since $d_H (X_k,\,X)\leq \frac{1}{m^{2p}}$, it follows that there is some $y\in X$ such that $d(x_D,\,y)\leq \frac{1}{m^{2p}}$. Therefore, there is an adjacent cell $D'$ to $D$ in $\mathcal{D}_{m^p}$ such that $y\in D'$. Thus, for every $D\in \mathcal{D}_{m^p}$ such that $D\cap X_k \neq \emptyset$  there is some $D'\in \mathcal{D}_{m^p}$ that is adjacent to $D$ such that $D'\cap X \neq \emptyset$. This gives us a mapping\footnote{Notice that $T$ depends on the choice of $x_D\in X_k\cap D$, but we suppress this in our notation.} 
$$T: \lbrace D\in \mathcal{D}_{m^p}: D\cap X_k \neq \emptyset \rbrace \rightarrow  \lbrace D\in \mathcal{D}_{m^p}: D\cap X \neq \emptyset \rbrace$$
This map may very well not be one-to-one, since there might be different $D\in \mathcal{D}_{m^p}$ that have mutual adjacent cells in $\mathcal{D}_{m^p}$. However, this map is (at most) $9$-to-$1$, since every cell has at most $9$ adjacent cells. Thus, 
$$N(X,\,\mathcal{D}_{m^p}) = \left| \lbrace D\in \mathcal{D}_{m^p}:\, D\cap X \neq \emptyset \rbrace \right| \geq \left| \text{Range} (T)\right|$$
$$ \geq \frac{1}{9} \left| \lbrace D\in \mathcal{D}_{m^p}:\, D\cap X_k \neq \emptyset \rbrace  \right| = \frac{1}{9}\cdot N(X_k,\, \mathcal{D}_{m^p}).$$
This is the claimed inequality.
\end{proof}

Finally, we recall the following standard fact about the star dimension of slices through Bedford-McMullen carpets:
\begin{Proposition} \label{Prop upper bound}
Let $F$ be a Bedford-McMullen carpet  with exponents $m>n$ such that $P_2(F)=[0,1]$. Then for all $(u,t)\in \mathbb{R}^2$ we have 
$$\dim^* \ell_{u,t} \cap F\leq \dim^* F-1.$$
\end{Proposition}
Proposition \ref{Prop upper bound} is a simple consequence of the following two  facts: Let $F$ be a carpet as in the Proposition. First, microsets of $F$ are product sets with the marginals being $P_2(F)=[0,1]$ and  $\ell_{0,t}\cap F$ for some $t\in [0,1]$ (see e.g. \cite{algom2016self, kaenmaki2016rigid}). Secondly, it follows from the work of Mackay \cite{mackay2011assouad} that $\dim_H  \ell_{0,t}\cap F = \dim^* F-1$ for some $t\in [0,1]$. See e.g. \cite[proof of Theorem 1.2]{algom2018slicing} for a closely related argument that combines these two facts to study $\dim^* \ell_{u,t} \cap F$.

\subsection{Perturbing the intercept of slices through carpets }
We begin with a definition: For $\kappa\in \mathbb{R}$, let $\pi_{\kappa} :\mathbb{R}^2 \rightarrow \mathbb{R}$ denote the orthogonal projection
\begin{equation*}
\pi_\kappa (x,y) = y+\kappa\cdot x.
\end{equation*}
In particular, $\pi_0 = P_2 $. 

Let $F$ be a Bedford-McMullen carpet as in Theorem \ref{Main Theorem}, with exponents $m>n$. In particular, recall that our assumptions on $F$ implies that there exists $x_0\in [0,1]$ such that 
$$ \lbrace x_0 \rbrace \times [0,1] \subseteq F.$$

We start with the following observation.
\begin{Lemma}\label{lemma:projection-condition-1}
Let $F$ be a Bedford-McMullen carpet as in Theorem \ref{Main Theorem}.
\begin{itemize}
\item[(1)] We have $d_H(\pi_\kappa(F),\pi_0(F))\to 0$ as $|\kappa|\to0$.
\item[(2)] There exists $c_0>0$ such that whenever $|\kappa|\le c_0$,  the projection $\pi_\kappa(F)$ is an interval.  
\end{itemize}
In particular, when $\kappa$ is small we have $|\pi_\kappa(F)|=1-o(1)$.
\end{Lemma}
\begin{proof}
Part (1) is simple a consequence of the fact that for each $x\in \mathbb{R}^2$, the function $\kappa\mapsto \pi_\kappa(x)$ is continuous.

Let us  now turn to Part (2).  We assume, without loss of generality, that $\kappa\le 0$; the case when $\kappa\ge 0$ can be treated in exactly the same way.   Denote $L=\{x_0\}\times [0,1]$.  Let us fix $\kappa\le 0$ and pick $(i_1,j_1), (i_2,j_2)\in D$ such that 
\[
\min\{ x: x\in\pi_\kappa(f_{i_1,j_1}(L))\}=\min\{x:x\in\pi_\kappa(f_{i,j}L)),\quad  (i,j)\in D\}
\]
and
\[
\max\{ x: x\in\pi_\kappa(f_{i_2,j_2}(L))\}=\max\{x:x\in\pi_\kappa(f_{i,j}(L)),\quad  (i,j)\in D\}.
\]
The choices of $(i_1,j_1)$ and $(i_2,j_2)$ might be non-unique.  Since  $F$ is self-affine it  follows that 
\begin{equation}\label{eq:lemma:projection-condition-1-1}
\pi_\kappa(f_{(i_1,j_1)^\infty})=\min\{x:x\in\pi_\kappa(F)\}
\textrm{ and } 
\pi_\kappa(f_{(i_2,j_2)^\infty})=\max\{x:x\in\pi_\kappa(F)\}.
\end{equation}

It is clear that there exists $c_0>0$ such that whenever $-c_0\le \kappa\le 0$,  for every $(i,j)\in D$
\begin{equation}\label{eq:lemma:projection-condition-1-2}
\pi_\kappa(f_{i,j}(L))\cap \pi_\kappa(L)\neq \emptyset.
\end{equation}
From now on, we assume that $-c_0\le \kappa\le 0$. Since $F$ is self-affine,  for each  $k\ge 0$,  
$$\pi_\kappa(f_{(i_1,j_1)^{k+1}}(L))\cap \pi_\kappa(f_{(i_1,j_1)^{k}}(L))$$
is a rescaled and translated copy  of 
$$\pi_{\kappa\frac{n^k}{m^k}}(f_{(i_1,j_1)}(L))\cap \pi_{\kappa\frac{n^k}{m^k}}(L).$$ 
Thus, for all $k\geq 0$ 
\begin{equation}\label{eq:lemma:projection-condition-1-3}
\pi_\kappa(f_{(i_1,j_1)^{k+1}}(L))\cap \pi_\kappa(f_{(i_1,j_1)^{k}}(L))\neq \emptyset.
\end{equation}
The same holds true if we replace $(i_1,j_1)$ by $(i_2,j_2)$:
\begin{equation}\label{eq:lemma:projection-condition-1-4}
\pi_\kappa(f_{(i_2,j_2)^{k+1}}(L))\cap \pi_\kappa(f_{(i_2,j_2)^{k}}(L))\neq \emptyset.
\end{equation}
It now follows that $\pi_\kappa(F)$ is an interval, using \eqref{eq:lemma:projection-condition-1-1},   \eqref{eq:lemma:projection-condition-1-3},   \eqref{eq:lemma:projection-condition-1-4} and the following facts:
\[
\bigcup_{k\ge 0}\pi_\kappa(f_{(i_1,j_1)^{k}}(L))\subset \pi_\kappa(F)  \textrm{ and }  \bigcup_{k\ge 0}\pi_\kappa(f_{(i_2,j_2)^{k}}(L))\subset \pi_\kappa(F).
\]

Finally, we have seen that when $|\kappa|$ is small enough,  $\pi_\kappa(F)$ is an interval.  Hence by (1),  we have $|\pi_\kappa(F)|=1- o(1)$
\end{proof}

In the following Claim  we use Lemma \ref{lemma:projection-condition-1} to show that: If for some $t$ and $Q\in \mathcal{D}_{m^k}$ we have $\ell_{u,t} \cap F\cap D \neq \emptyset$,  then, given $b\in \mathbb{N}$,  if $t'$ and $t$ are very close then
$$ N\left(\ell_{u,t} \cap F\cap Q, \, \mathcal{D}_{m^{k+b}} \right) \approx  N\left(\ell_{u,t'} \cap F\cap Q, \, \mathcal{D}_{m^{k+b}} \right).$$

\begin{Claim}\label{Lemma 2}
There exists an absolute constant $C_1=C_1(F)>0$ depending only on $F$  and  $c_0$ (the constant from Lemma \ref{lemma:projection-condition-1}) such that the following holds: Let $0<u\le c_0$, $t\in \mathbb{R}$,  $k\in \mathbb{N}$ be such that
\begin{equation*}
\ell_{u,t} \cap F\cap Q \neq \emptyset, \text{ where } Q\in \mathcal{D}_{m^k}.
\end{equation*}
 Then  for every $b\in \mathbb{N}$ there exists $\delta=\delta(c_0,k,b,t) >0$  such that at least one of the following alternatives hold true:
\begin{enumerate}
\item For every $t'\in [t,t+\delta]$ we have
\begin{equation*}
N(\ell_{u,t'} \cap F\cap Q, \quad \mathcal{D}_{m^{k+b}} )\geq C_1 \cdot \left( N(\ell_{u,t} \cap F\cap Q, \quad  \mathcal{D}_{m^{k+b}})-1\right).
\end{equation*}

\item For every $t'\in [t-\delta,t]$ we have
\begin{equation*}
N(\ell_{u,t'} \cap F\cap Q, \quad \mathcal{D}_{m^{k+b}} )\geq C_1 \cdot \left( N(\ell_{u,t} \cap F\cap Q, \quad  \mathcal{D}_{m^{k+b}}) -1\right).
\end{equation*}
\end{enumerate}
\end{Claim}

\begin{proof}
Recall that $D$ is the digit set associated with the carpet $F$.
Note that for any $I\in D^{k+b}$, $0<u\le c_0$, and $t\in \mathbb{R}$, 
\[
N\left(\ell_{u,t}\cap f_{I}(F),\,\mathcal{D}_{m^{k+b}}\right) \le C', \text{ where } C'=C'(c_0).
\]
So, for every $Q\in \mathcal{D}_{m^{k}}$
\begin{equation}\label{eq:proof-Lemma 2 -1}
N(\ell_{u,t} \cap F\cap Q, \quad  \mathcal{D}_{m^{k+b}})\le C'\left|\left\{I\in D^{k+b}:\, \ell_{u,t} \cap f_I(F) \cap Q \neq \emptyset\right\}\right|.
\end{equation}
On the other hand, for every $I\in D^{k+b}$ such that $\ell_{u,t} \cap f_I(F) \cap \text{Int}( Q) \neq \emptyset$ we may  associate a cell $D\in \mathcal{D}_{m^{k+b}}, D\subset Q$ with
$$\ell_{u,t} \cap F\cap \text{Int}( Q) \cap D \neq \emptyset.$$
Indeed, since $f_I(F)\cap \text{Int}( Q)$ is included in a column of such cells, once such choice is the lowest $D$ in this column that intersects $\ell_{u,t} \cap F\cap \text{Int}( Q)$. This map is (at most) $2$-to-$1$, and since $\left| \ell_{u,t} \cap F\cap \partial Q \right| \leq 4$, it follows that
\begin{equation}\label{eq:proof-Lemma 2 - the other ineqaulity}
\left|\left\{I\in D^{k+b}:\, \ell_{u,t} \cap f_I(F) \cap Q \neq \emptyset\right\}\right|\leq 8\cdot \left( N(\ell_{u,t} \cap F\cap Q, \quad  \mathcal{D}_{m^{k+b}})+1 \right).
\end{equation}

In the following, we fix $0<u\le c_0$. For any $I\in D^{k+b}$, 
$$\pi_{-u} (f_I(F)) = \text{ rescaled and translated copy of } \pi_{-u\frac{n^{k+b}}{m^{k+b}}} (F).$$ So, it follows from Lemma \ref{lemma:projection-condition-1}
 that  $\pi_{-u} (f_I(F))$ is an interval.  So, if $f_I(F)\cap \ell_{u,t}\neq \emptyset$,  letting $\delta':=|\pi_{-u} (f_I(F))|/2$, 
$$  f_I(F)\cap \ell_{u,t'}\neq \emptyset  \text { for all } t'\in (t,t+\delta'), \text{ or } f_I(F)\cap \ell_{u,t'}\neq \emptyset \text { for all } t'\in (t-\delta',t).$$
In view of this,  we then deduce that there exists $\delta>0$ depending on $k,b$ and $c_0$,  such that whenever $f_I(F)$ intersects the interior of $Q\in \mathcal{D}_k$ and $\ell_{u,t}\cap f_{I}(F)\cap Q\neq \emptyset$, then 
\begin{equation}\label{eq:proof-Lemma 2 -2}
 \ell_{u,t'}\cap f_{I}(F)\cap Q\neq \emptyset \text { for all }  t'\in (t-\delta,t) \textrm{ or } \ell_{u,t'}\cap f_{I}(F)\cap Q\neq \emptyset \text { for all }  t'\in (t,t+\delta).
\end{equation}
On the other hand,  it is readily checked that for $u\le c_0$ and $t\in\mathbb{R}$,  there exists at most two $I\in D^{k+b}$ such that $f_I(F)$ doesn't intersect the interior of $Q\in \mathcal{D}_k$ and $\ell_{u,t}\cap f_{I}(F)\cap Q\neq \emptyset$.  Thus,  combining \eqref{eq:proof-Lemma 2 -1}, \eqref{eq:proof-Lemma 2 - the other ineqaulity}, and \eqref{eq:proof-Lemma 2 -2}, we obtain the desired conclusion.

\end{proof}

\subsection{Perturbing the slope of  horizontal slices through carpets} \label{Section slope}

Let $0\leq j \leq n-1$ and let $t=t_j\in P_2 (F)=[0,1]$ be the point
\begin{equation*}
t:=\sum_{k=1} ^\infty \frac{j}{n^k}.
\end{equation*} 
It follows from the work of Mackay \cite{mackay2011assouad} that we can select $j$ so that, if $D_j = \lbrace i:\, (i,j) \in D\rbrace$ and
$$ \ell_{0,t}\cap F =\left\lbrace \sum_{k=1} ^\infty \frac{x_k}{m^k}: x_k \in D_j \right\rbrace, $$
we have $\dim_H  \ell_{0,t}\cap F = \dim^* F-1$.

\begin{Claim} \label{Claim hori line}
For every $p\in \mathbb{N}$ there exists $\delta=\delta(p,t)>0$ such that:

For every $u\in [0,\delta)$ there are $t_u$ and $\delta'>0$ (depending  on all previous parameters) such that for every $t\in [t_u-\delta',t_u+\delta']$, 
\begin{equation*}
N(\ell_{u,t}\cap F , \, \mathcal{D}_{m^p} ) \geq \frac{1}{9}   N(\ell_{0,t_j}\cap F , \, \mathcal{D}_{m^p}).
\end{equation*}
\end{Claim}

\begin{proof}
Let $(x,t)\in  F\cap \ell_{0,t}\neq \emptyset$ and let $D\in \mathcal{D}_{m^p}$ be such that $(x,t) \in D$. Let $k\in \mathbb{N}$ be such that $\frac{1}{n^k} < \frac{1}{m^p}$. Find a cylinder $(x,t)\in \phi_{I_D} (F)$ where $|I_D|=k$ such  that $\phi_{I_D}$ corresponds to the first $k$ digits in an expansion of $x$ in base $m$, and to the first $k$ digits in the expansion of $t$ in base $n$.   Then $ \phi_{I_D} (F)\cap D \neq \emptyset$.

Let $|\kappa|\leq 1$  so that $|\frac{n^k}{m^k}\cdot \kappa|<1$. Recall that we parametrize the projections $\mathbb{R}^2 \rightarrow \mathbb{R}$ by $\pi_\kappa (x,y)=y+\kappa\cdot x$. Then  $\pi_\kappa ( \phi_{I_D} (F))$ is a translation of the set
\begin{equation*}
\pi_\kappa \left( \diag \left(\frac{1}{m^k},\frac{1}{n^k}\right) \cdot F\right)=\frac{1}{n^k}\pi_{\kappa\frac{n^k}{m^k}}(F). 
\end{equation*}
Namely, by the definition of the cylinder $I_D$ and of $\pi_\kappa$,
$$\pi_\kappa ( \phi_{I_D} (F)) = \frac{1}{n^k}\pi_{\kappa\frac{n^k}{m^k}}(F)+ \sum_{i=1} ^{k} \frac{j}{n^i}+\kappa\cdot P_1\phi_{I_D}(0).$$
Recall that by Lemma \ref{lemma:projection-condition-1},  $\pi_{\kappa\frac{n^k}{m^k}}(F)$ is an interval when $\kappa\frac{n^k}{m^k}$ is small enough.

We claim that there is some small $\delta=\delta(p,t)>0$ such that for every $\kappa\in (-\delta,0)$:
\begin{equation*}
\bigcap_{D\in \mathcal{D}_{m^p}:\, D\cap\ell_{0,t}\neq \emptyset} \pi_\kappa \phi_{I_D} (F)  \text{ contains a non trivial open interval.}
\end{equation*}
Indeed, let $J:= \frac{1}{n^k}\pi_{\kappa\frac{n^k}{m^k}}(F)+ \sum_{i=1} ^{k} \frac{j}{n^i}$, and let $t_{D}:=P_1\phi_{I_D}(0)$. Then
$$ \bigcap_{D\in \mathcal{D}_{m^p}:\, D\cap\ell_{0,t}\neq \emptyset} \pi_\kappa \phi_{I_D} (F) = \bigcap_{D\in \mathcal{D}_{m^p}:\, D\cap\ell_{0,t}\neq \emptyset} \left( J+ \kappa\cdot t_D \right).$$
Since $k$ is fixed and $|J|=\frac{1}{n^k}(1-o(1))$ by Lemma \ref{lemma:projection-condition-1}, the Claim follows since we may take $\kappa \ll \frac{1}{n^k}$, and since $0\leq t_D \leq 1$.

Finally, if 
$$t' \in \bigcap_{D\in \mathcal{D}_{m^p}:\, D\cap\ell_{0,t}\neq \emptyset} \pi_\kappa \phi_{I_D} (F)$$
then for every $D\in \mathcal{D}_{m^p}$ such that $D\cap\ell_{0,t}\neq \emptyset$,  $\ell_{-\kappa, t'}$ intersects $\phi_{I_D} (F)$. Therefore, $\ell_{-\kappa, t'}\cap F$ intersects a cell that is adjacent to $D$ in $\mathcal{D}_{m^p}$.   This proves the Claim since there are at most $9$ such cells for every $D$.
\end{proof}

\subsection{Proof of Theorem \ref{Main Theorem}} 

Let $t\in P_2(F)$ be as in the beginning of Section \ref{Section slope} 
$$t= \sum_{k=1} ^\infty \frac{j}{n^k},$$
and consider the line $\ell_{0,t}$, that intersects $F$ with Hausdorff dimension $\dim^* F -1$. The following Claim is the key construction of this paper:
\begin{Claim} \label{Key claim}
For every $i\in \mathbb{N}$  there are  integers $k_i\geq 0$, cells $D_i \in \mathcal{D}_{m^{k_i}}$ and lines $\ell_{u,t_i}$ with the following properties:
\begin{itemize}
\item All the lines $\ell_{u,t_i}$ have the same slope $u \neq 0$. Every small enough $u$ may be chosen to be this common slope.

\item There exists some global $C'=C'(u)>0$ such that for every $i$,
\begin{equation*}
N\left(\ell_{u,t_i} \cap F \cap D_i,\, \, \mathcal{D}_{m^{k_i+i}}\right)\geq m^{(\dim^* F -1)\cdot i} \cdot C'.
\end{equation*}

\item For every $p\geq 2$ and every $1\leq i <p$, $t_p$  lies in the \textit{interior} of a "good" one sided neighbourhood of $t_i$ in the sense of Claim \ref{Lemma 2}. Therefore, by Claim \ref{Lemma 2},
\begin{equation*}
N(\ell_{u,t_{p}}\cap F \cap D_i,\,  \, \mathcal{D}_{m^{k_i+i}}) \geq \frac{1}{C_1} \cdot \left(  N(\ell_{u,t_i}\cap F \cap D_i,\, \, \mathcal{D}_{m^{k_i+i}}) - 1 \right)  
\end{equation*}
for some absolute constant $C_1=C_1(F)$ only depending on $F$, as in Claim \ref{Lemma 2}.
\end{itemize}
\end{Claim}

The construction is inductive:
$$ $$

\noindent{ \textbf{The initial step.}} Recall that
\begin{equation*}
N(\ell_{0,t} \cap F, \, \mathcal{D}_{m})  =m^{(\dim^* F -1)\cdot 1}.
\end{equation*}
By Claim \ref{Claim hori line} there exists $\delta>0$ such that for every slope $u\in [0,\delta)$ there  are $t_u$ and $\delta'>0$ such that  for every $t\in [t_u-\delta' ,t_u+\delta']$ , we have
$$N( F\cap \ell_{u,t},\, \mathcal{D}_{m}) \geq \frac{1}{9}\cdot  m^{(\dim^* F -1)\cdot 1}.$$
Pick a pair $(u,t)$ where both $u$ and $t$ are \textit{interior} points in these corresponding neighbourhoods, and $u<c_0$ so that Claim \ref{Lemma 2} applies. This will be $\ell_{u,t_1}$. The slope $u$ will be the slope of all the lines that we construct later on. We also select $D_1 = [0,1]$ and $k_1 = 0$.

$$ $$
\noindent{ \textbf{Inductive step: Construction of the line $\ell_{u,t_p}.$}} Assume we have found  $p$ integers $k_1,...,k_{p-1}$ and $p-1$ lines $\ell_{u,t_1},...,\ell_{u,t_{p-1}}$ such that for every $1\leq i < p$ there is a cell $D_i \in \mathcal{D}_{m^{k_i}}$ such that for some global $C'>0$ (that we will discuss later), the properties in the Claim hold true. In particular, the successive distances between  $\ell_{t_i}$ and $t_{i+1}$ go to zero very fast.

Apply Claim \ref{Claim hori line} for $p$ to find $\delta>0$ so that for any $u'\in [0, \delta)$ there is are $t_{u'}$ and $\delta'>0$ such that  for every $t \in [t_{u'}-\delta',t_{u'}+\delta']$  we have 
\begin{equation} \label{Eq number}
N( F\cap \ell_{u',t},\, \mathcal{D}_{m^{p}}) \geq \frac{1}{9}\cdot  m^{(\dim^* F -1 )\cdot p}.
\end{equation}

Let $D_{p-1} \in \mathcal{D}_{m^{k_{p-1}}}$ be the cell we have previously constructed.  By our assumption, $t_{p-1}$ lies in the \textit{interior} of a "good" one sided neighbourhood of $t_i$ for every $1\leq i \leq p-1$, in the sense of Claim \ref{Lemma 2}. Suppose without the loss of generality that a "good" neighbourhood of $t_{p-1}$ (good in the sense of Claim \ref{Lemma 2}) has the form $[t_{p-1},t_{p-1}+\delta'']$.   Then, by  Claim \ref{Lemma 2}, for every $s\in (t_{p-1},t_{p-1}+\delta'')$ we have $\ell_{u,s}\cap F \cap D_{p-1} \neq \emptyset$.  Select $s$ to  be sufficiently close to $t_{p-1}$ so that for every $1\leq i <p$, $s$ \textit{remains} in a good one sided neighbourhood of $t_{i}$ in the sense of Claim \ref{Lemma 2}.

Let $z'\in \ell_{u,s}\cap F \cap D_{p-1}$. Consider all the cylinder maps $\phi_I$ such that:
\begin{enumerate}
\item $z' \in \phi_I (F)$.

\item  $|I|=k$ satisfies that $\frac{n^k}{m^k}\cdot u < \delta$.

\end{enumerate}
Let $\ell_{\frac{n^k}{m^k}\cdot u,t}$ be any  line satisfying \eqref{Eq number}.  Then for any $I$ the slope of $\phi_I (\ell_{\frac{n^k}{m^k}\cdot u,t})$  is $u$. Since the line $\phi_I (\ell_{\frac{n^k}{m^k}\cdot u,t})$ passes through a point that can be made arbitrarily close to $z'$ (by taking the generation of $I$ to be even larger), its intercept can be made to be arbitrary close to the intercept of the line of slope $u$ through $z'$. This line is exactly $\ell_{u,s}$, and we have chosen its intercept $s$ so that $s\in (t_{p-1},t_{p-1}+\delta'')$. Thus, we choose $t_p$ to be the intercept of $\phi_I (\ell_{\frac{n^k}{m^k}\cdot u,t})$, where $I$ is chosen so that the intercept of $\phi_I (\ell_{\frac{n^k}{m^k}\cdot u,t})$ is in a good one sided neighbourhood of  $t_{p-1}$. Furthermore, by making the distance between them even smaller, it is in good neighbourhoods of all the preceding $t_i$'s. Let $k_p:=k$ be the generation of this cylinder $I$.

$$ $$
\noindent{ \textbf{Construction of the cell $D_p$}}   We claim that there is a constant $C=C(u)>0$ (the same constant that works for the previous lines) such that for some $D\in \mathcal{D}_{m^{k_p}}$ 
\begin{equation*}
N\left(\ell_{u, t_p}\cap F \cap D, \, \mathcal{D}_{m^{k_{p}+p}}\right)\geq m^{(\dim^* F -1)\cdot p} \cdot C.
\end{equation*} 
Indeed, Let $\tilde{D} = \phi_I ([0,1]^2)$ be the corresponding element of the partition $\mathcal{D}_{m^{k_p}}\times \mathcal{D}_{n^{k_p}}$. 
For every $u\neq 0$ we define
$$s(u) = \max_{p\in \mathbb{N},\, t\in \mathbb{R}, D'\in \mathcal{D}_{m^p} } \left| \left\lbrace D\in \mathcal{D}_{m^p}: \,D\cap \ell_{u,t} \neq \emptyset \text{ and } P_1 \left( D \right) = P_1 \left( D' \right) \right\rbrace \right|.$$
That is, $s(u)$ is the maximal possible number of cells in the same column of $\mathcal{D}_{m^p}$ that any line with slope $u$ can intersect (across all $p$).  
Then, since $\phi_I$ is invertible,
\begin{eqnarray*}
N(\phi_I^{-1} (\ell_{u,t_p}) \cap F\cap [0,1]^2,\, \mathcal{D}_{m^p} ^2) &\leq &  s(u)\cdot  N( \phi_I^{-1} (\ell_{u,t_p}) \cap F \cap [0,1]^2,\quad \mathcal{D}_{m^p} \times [0,1]) \\
&=&s(u)\cdot  N(\ell_{u,t_p} \cap \phi_I(F) \cap \tilde{D}, \quad \mathcal{D}_{m^{k_p+p}} \times \mathcal{D}_{n^{k_p}})\\
&\leq & s(u)\cdot  N(\ell_{u,t_p} \cap F \cap \tilde{D}, \quad \mathcal{D}_{m^{k_p+p}} \times \mathcal{D}_{n^{k_p}}) \\
&\leq & 2\cdot s(u)\cdot  N(\ell_{u,t_p} \cap F \cap \tilde{D}, \quad \mathcal{D}_{m^{k_p+p}} ^2 ).
\end{eqnarray*}
For the last inequality, observe that any cell in the partition $\mathcal{D}_{m^{k_p+p}} ^2$  can intersect at most two cells in the partition $\mathcal{D}_{m^{k_p+p}} \times \mathcal{D}_{n^{k_p}}$.

Finally, let us partition $\tilde{D} \in \mathcal{D}_{m^{k_p}} \times \mathcal{D}_{n^{k_p}}$ into  $\mathcal{D}_{m^{k_p}} ^2$ cells (which are squares of side $m^{-k_p}$ that sit  one above the other). Then $\ell_{u, t_p}\cap F$ can only intersect at most $s(u)$ such cells. Therefore, there is at least one cell $D\in \mathcal{D}_{m^{k_p}}$ such that
\begin{equation*}
N(\ell_{u,t_p} \cap F \cap \tilde{D}, \, \mathcal{D}_{m^{k_p+p}} ^2) \leq s(u)\cdot  N(\ell_p \cap F \cap D, \,\mathcal{D}_{m^{k_p+p}} ^2).
\end{equation*}
We pick this cell as our $D_p$. Recalling \eqref{Eq number}, the constant $C$ is thus $s(u)^{-2}\cdot \frac{1}{2\cdot 9}$. The proof of the Claim is complete. \hfill{$\Box$}
$$ $$

\noindent{ \textbf{Proof of Theorem \ref{Main Theorem}}} With the notations of Claim \ref{Key claim},  consider the sequence of compact sets $X_p:=\ell_{u,t_p} \cap F$. Then, upon moving to a subsequence, the sequence $X_p$ has a Hausdorff metric limit $A$. It is easy to see that $A$ is a subset of $\ell_{u,t}\cap F$ for some $t$. Recall that $D_i \in \mathcal{D}_{m^{k_i}}$ are the cells we have constructed in Claim \ref{Key claim}.

Fix $i\in \mathbb{N}$. Then, up to taking a further subsequence,  $\ell_{u,t_p}\cap F \cap \overline{D_i}$ converges to $A_i$, where 
$$A_i \subseteq A\cap \overline{D_i}\subseteq \ell_{u,t} \cap F \cap \overline{D_i}, \text{ as  } p\rightarrow \infty.$$
Thus,  for every $p$ large enough in our subsequence, by Lemma \ref{Lemma 0} and our construction,
\begin{eqnarray*}
 N(\ell_{u,t}\cap F \cap \overline{D_i},\quad \mathcal{D}_{m^{k_i+i}}) &\geq& N(A_i,\quad \mathcal{D}_{m^{k_i+i}})\\
 &\geq& \frac{1}{9}\cdot  N(\ell_{u,t_p}\cap F \cap \overline{D_i},\quad \mathcal{D}_{m^{k_i+i}})\\
 &\geq& \frac{1}{9}\cdot N(\ell_{u,t_p}\cap F \cap D_i,\quad \mathcal{D}_{m^{k_i+i}})\\
 &\geq& \frac{1}{9}\cdot  \frac{1}{C_1} \cdot \left( N(\ell_{u,t_i}\cap F \cap D_i,\quad \mathcal{D}_{m^{k_i+i}})-1\right)\\
 &\geq& \frac{1}{9}\cdot \frac{1}{C_1} \cdot \frac{1}{C'}\cdot \left(  m^{(\dim^* F -1)\cdot i}-1)\right)
\end{eqnarray*}
Also, since $\ell_{u,t}$ is a line with slope $0<|u|<\infty$, there is a global constant $C=C(u)$ such that for every large enough $i$,
\begin{equation*}
 N(\ell_{u,t}\cap F \cap D_i,\quad \mathcal{D}_{m^{k_i+i}})  \geq C\cdot \left(  N(\ell_{u,t}\cap F \cap \overline{D_i},\quad \mathcal{D}_{m^{k_i+i}}) - 1\right).
\end{equation*}
The explanation is similar to the arguments  given in  the proof of Claim \ref{Lemma 2}.

We conclude that there are uniform constants $C_3,C_4>0$ such that for every $i$ there exists $k_i$ and $D_i \in \mathcal{D}_{m^{k_i}}$ with
$$N(\ell_{u,t}\cap F \cap D_i,\quad \mathcal{D}_{m^{k_i+i}})\geq C_3\cdot \left(  m^{(\dim^* F -1)\cdot i}\right)-C_4.$$
Putting this into Theorem \ref{Furstneberg Formula}, we see that
\begin{equation*}
\dim^* \ell_{u,t} \cap F = \lim_{i\rightarrow \infty} \max_{k} \lbrace \frac{\log N(\ell_{u,t} \cap F \cap D, \quad  \mathcal{D}_{m^{k+i}})}{i\log m} : D\in \mathcal{D}_{m^k} \rbrace \geq \dim^* F - 1.
\end{equation*}
Via Proposition \ref{Prop upper bound}, this completes the proof of Theorem \ref{Main Theorem} for small slopes $u>0$ in the sense of Claim \ref{lemma:projection-condition-1}. Arguing in a completely analogues manner for negative $u$'s, there is a left one sided open neighbourhood of $0$ of slopes such that the Theorem holds. Thus, there is an open ball of such slopes, concluding the proof.  \hfill{$\Box$}

\bibliography{bib}{}
\bibliographystyle{plain}

\end{document}